\documentclass[reqno]{amsart}
\usepackage{amsmath, amsthm, amscd, amssymb, amsfonts, amsbsy}
\usepackage{latexsym, color, enumerate}
\usepackage{pxfonts}
\usepackage{mathrsfs}
\usepackage[dvipsnames]{xcolor}
\usepackage{bbm}

\usepackage{todonotes}
\usepackage{marginnote}

\theoremstyle{plain}
\newtheorem{theorem}[equation]{Theorem}
\newtheorem{lemma}[equation]{Lemma}

\newtheorem{proposition}[equation]{Proposition}

\theoremstyle{definition}
\newtheorem{definition}[equation]{Definition}
\newtheorem{condition}[equation]{Condition}

\theoremstyle{remark}
\newtheorem{remark}[equation]{Remark}

\newtheorem{conclusion}[equation]{Conclusion}

\newcommand{\dv}{\operatorname{div}}

\newcommand{\supt}{\operatorname{supt}}

\numberwithin{equation}{section}

\newcommand{\bR}{\mathbb{R}}

\providecommand{\set}[1]{\{#1\}}

\providecommand{\abs}[1]{\lvert#1\rvert}
\providecommand{\Abs}[1]{\left\lvert#1\right\rvert}

\providecommand{\norm}[1]{\lVert#1\rVert}

\renewcommand{\vec}[1]{\boldsymbol{#1}}

\def\dVert{\,\,\text{--}\kern-.46em\|}

\providecommand{\dashnorm}[1]{\dVert#1\Vert}

\begin{document}
\title[Harnack inequality]
{Harnack inequality for parabolic equations in double-divergence form with singular lower order coefficients}

\author[I. Gy\"ongy]{Istvan Gy\"ongy}
\address[I. Gy\"ongy]{School of Mathematics, University of Edinburgh, James Clerk Maxwell Building, Peter Guthrie Tait Road, Edinburgh, EH9 3FD, United Kingdom}
\email{I.Gyongy@ed.ac.uk}

\author[S. Kim]{Seick Kim}
\address[S. Kim]{Department of Mathematics, Yonsei University, 50 Yonsei-ro, Seodaemun-gu, Seoul 03722, Republic of Korea}
\email{kimseick@yonsei.ac.kr}
\thanks{S. Kim is supported by the National Research Foundation of Korea (NRF) under agreement NRF-2022R1A2C1003322.}

\subjclass[2010]{Primary 35B45, 35B65; Secondary 35K10}

\keywords{Harnack inequality, parabolic equations, double divergence form}

\begin{abstract}
This paper investigates the Harnack inequality for nonnegative solutions to second-order parabolic equations in double divergence form. We impose conditions where the principal coefficients satisfy the Dini mean oscillation condition in $x$, while the drift and zeroth-order coefficients belong to specific Morrey classes.
Our analysis contributes to advancing the theoretical foundations of parabolic equations in double divergence form, including Fokker-Planck-Kolmogorov equations for probability densities.
\end{abstract}
\maketitle

\section{Introduction and main results}

We consider a parabolic operator in double divergence form
\begin{equation*}			
\mathscr{L}^*u:=-\partial_t u + \dv^2 (\mathbf A u) -\dv(\vec b u) = -\partial_t u + D_{ij}(a^{ij} u) - D_i(b^i u),
\end{equation*}
where $\mathbf A=(a^{ij})$ and $d \times d$ symmetric matrix valued function and $\vec b=(b^1, \ldots, b^d)$ is a vector valued function defined on $\bR^{d+1}=\bR\times \bR^d$.
Here, the usual summation convention is adopted.
We note that the operator $\mathscr{L}^*$ is the formal adjoint of $\mathscr{L}$, where
\[
\mathscr{L} v := \partial_t v + a^{ij}D_{ij}v + b^i D_i v.
\]

Let $c$ be a real valued function on $\bR^{d+1}$.
In this article, we are interested in Harnack's inequality for nonnegative solutions $u$ to
\[
\mathscr{L}^*u-cu=0.
\]

We assume that the principal coefficients matrix $\mathbf A=(a^{ij})$ is symmetric and satisfy the uniform ellipticity condition.
We also assume that $\mathbf A$ is of Dini mean oscillation in $x$.
This condition is stronger than $\mathbf A$ belonging  to $\mathrm{VMO}_x$ but is weaker than $\mathbf A$ being Dini continuous in $x$.
We allow the coefficient $b^i$ and $c$ to be singular.
Refer to Section~\ref{sec2} for the precise conditions on the coefficients.

An important example of parabolic equations in double divergence form is parabolic Fokker-Planck-
Kolmogorov equations for densities.
Interested readers are asked to refer to an excellent survey book on this subject \cite{BKRS15}.

The main theorem of this paper is as follows.
\begin{theorem}		\label{thm_main}
Assume Conditions \ref{cond1} and \ref{cond2} hold.
Let $R>0$ be a fixed number, $0<r<R/4$, and $(t_0, x_0) \in \bR^{d+1}$.
Denote $C_{r}=(t_0-r^2, t_0] \times B_{r}(x_0)$. 
Suppose $u \in L_1(C_{4r})$ is a nonnegative solution of
\[
\mathscr{L}^* u - cu =0\;\;\text{ in }\;\; C_{4r}.
\]
Then, we have
\begin{equation}			\label{harnack}
\sup_{(t_0-3r^2, t_0-2r^2)\times B_{r/3}(x_0)} u \le N \inf_{(t_0-r^2, t_0)\times B_{r/3}(x_0)} u,
\end{equation}
where $N$ is a constant depending only on $d$, $\delta$, $\omega_{\mathbf A}^{\mathsf x}$, $p$, $q$, $\beta$, $\mathfrak{b}$, $\mathfrak{c}$, and $R$.
\end{theorem}

A few remarks are in order.
Harnack inequalities serve as crucial tools in the study of elliptic and parabolic 
partial differential equations of the second order.
For divergence form elliptic equations with measurable coefficients, 
Moser \cite{Moser61} proved the Harnack inequality for $W^{1}_{2}$ weak solutions, 
which he later extended to the parabolic setting \cite{Moser64}.
For non-divergence form parabolic equations with measurable coefficients, 
Krylov and Safonov \cite{KS79, KS80} established the Harnack inequality for $W^2_{d+1}$ strong solutions.
The elliptic counterpart of the Krylov-Safonov Harnack inequality is particularly discussed in \cite{Safonov80}.
In contrast to equations in divergence or non-divergence form, 
Harnack inequalities for double divergence form equations require more than a measurability condition.
This is mainly due to the fact that weak solutions to equations 
in double divergence form are not necessarily bounded 
when the coefficients $\mathbf A$ lack some control on the modulus of continuity.
A counterexample in the elliptic setting is provided in \cite{Bauman84b}.
However, if the coefficients $\mathbf{A}$ are H\"older or Dini continuous, 
weak solutions of double divergence form equations are locally bounded and even continuous.
See \cite{Sjogren73, Sjogren75} for elliptic equations in double divergence form 
with H\"older continuous coefficients.
Recently, it has been shown that if $\mathbf{A}$ has Dini mean oscillation, 
then weak solutions to double divergence form equations are continuous. 
Refer to \cite{DK17} for elliptic equations and \cite{DEK21} for parabolic equations.
The Harnack inequality for nonnegative solutions to the elliptic 
equation $\mathrm{div}^2(\mathbf{A} u)=0$ with $\mathbf A$ belonging 
to the Dini mean oscillation class was presented in \cite{DEK18}.
In the parabolic setting, a corresponding result was presented in \cite{DEK21}, 
albeit somewhat implicitly.
Recently, in \cite{BRS23}, the Harnack inequality was established 
for elliptic equations in the form $\mathrm{div}^2(\mathbf A u) -\mathrm{div}(\vec b u)+cu=0$, 
with $\mathbf A$ satisfying the Dini mean oscillation condition 
and $\vec b$, $c \in L_{p, \mathrm{loc}}$ for $p>d$.
Our main theorem, Theorem \ref{thm_main}, extends the result 
established in \cite{BRS23} to the parabolic context with less restrictive assumptions.
Notably, we improve upon the conditions required in \cite{BRS23} 
by only requiring $c \in L_{p, \mathrm{loc}}$ with $p>d/2$.
For detailed information, refer to Theorem \ref{thm2}.

It should be noted that there is another significant difference 
between equations in divergence or non-divergence form and those in double divergence form.
In the former case, the Harnack inequality implies H\"older continuity estimates for solutions.
However, in the latter case, H\"older estimates for solutions are not available in general.
This is essentially because the constant function is not a solution to the elliptic 
or the parabolic equation in double divergence form.

The organization of the paper is as follows:
In Section \ref{sec2}, we introduce notations and present some preliminary results.
In particular, we state the Harnack inequality for nonnegative solutions to parabolic equations in double divergence form without lower-order coefficients, which was not explicitly presented in \cite{DEK21}, as Theorem \ref{thm:harnack}.
In Section \ref{sec3}, we provide the proof of our main result, Theorem \ref{thm_main}, by breaking it into several steps.
The last section, Section \ref{sec4}, is a brief one devoted to the Harnack inequality for elliptic equations in double divergence form.
\section{Preliminaries}			\label{sec2}

We represent a point in $\mathbb{R}\times \mathbb{R}^d=\mathbb{R}^{d+1}$ as $X=(t,x)$, where $x = (x^1,\ldots, x^d)$ always denotes a point in $\mathbb{R}^d$.
The parabolic distance between the points $X=(t,x)$ and $Y=(s,y)$ in $\mathbb{R}^{d+1}$ is defined as
\[
\abs{X-Y}_{p}=\max(\sqrt{\abs{t-s}}, \abs{x-y}).
\]
Let $Q \subset \bR^{d+1}$. For $0<\alpha\le 1$, we define
\[
[u]_{\mathscr{C}^{\alpha/2,\alpha}(Q)} = \sup_{\substack{X, Y \in Q\\X\neq Y}}\frac{\abs{u(X)-u(Y)}}{\abs{X-Y}_{p}^\alpha}.
\]
Also, for a function $u$ having a spatial derivative $Du$, we define
\[
[u]_{\mathscr{C}^{(1+\alpha)/2,1+\alpha}(Q)}=[Du]_{\mathscr{C}^{\alpha/2,\alpha}(Q)}+\sup_{\substack{(t,x), (s,x) \in Q\\t\neq s}}\frac{\abs{u(t,x)-u(s,x)}}{\abs{t-s}^{(1+\alpha)/2}}.
\] 

For $p$, $q \in [1,\infty)$ and a domain $Q\subset \bR^{d+1}$, we denote by $L_{p,q}(Q)$ the space of measurable functions $f$ on $Q$ with
\[
\norm{f}_{L_{p,q}(Q)} = \norm{f \mathbbm{1}_Q}_{L_{p,q}} = \left(\int_{\bR} \left(\int_{\bR^d} \abs{f \mathbbm{1}_Q(t,x)}\,dx\right)^{q/p}\,dt\right)^{1/q}<\infty.
\]
When $p$ or $q=\infty$, we use essential supremum instead of the integral. 
We note that $L_{p,p}(Q)=L_p(Q)$, the usual Lebesgue class.

By $W^{1,2}_{p,q}(Q)$, we mean the collection of $u$ such that $\partial_t u$, $D^2u$, $Du$, $u \in L_{p,q}(Q)$.
The norm in $W^{1,2}_{p,q}(Q)$ is introduced in an obvious way.
We drop $Q$ if $Q=\bR^{d+1}$.

For $X=(t, x) \in \bR^{d+1}$ and $r>0$, we define $C_r(X)=(t-r^2, t]\times B_{r}(x)$, where $B_r(x)$ is the $d$ dimensional open ball with radius $r$ centered at $x$.
Let us also introduce the forward cylinder $\tilde C_r(X)=\tilde C_r(t,x)=[t, t+r^2)\times B_r(x)$ and let $\mathbb C_r$ be the collection of all cylinders $\tilde C_r(X)$ in $\bR^{d+1}$.

For $\beta \ge 0$, we define the Morrey space $E_{p,q,\beta}$ as in \cite{Krylov2023}, i.e., the set of all functions $g\in L_{p,q,\rm{loc}}$ such that
\begin{equation}			\label{eq_morrey}
\norm{g}_{E_{p,q,\beta}}:=\sup_{\rho \le 1, \;C \in \mathbb{C}_\rho} \rho^{\beta} \dashnorm{g}_{L_{p,q}(C)}<\infty,
\end{equation}
where $\dashnorm{g}_{L_{p,q}(C)}=\norm{1}_{L_{p,q}(C)}^{-1}\norm{g}_{L_{p,q}(C)}=N(d)\rho^{-d/p-2/q}\norm{g}_{L_{p,q}(C)}$ for $C \in \mathbb{C}_\rho$.
Define
\[
E^{1,2}_{p,q,\beta}=\set{u: u,\, Du,\, D^2 u,\, \partial_t u \in E_{p,q,\beta}}
\]
and equip $E^{1,2}_{p,q,\beta}$ with an obvious norm.

For $C=\tilde C_r(t_0,x_0) \in \mathbb C_r$, we set
\[
\bar{\mathbf A}^{\textsf x}_{C}(t)=\fint_{B_r(x_0)} \mathbf A(t,x)\,dx,
\]
and for $r>0$, we define
\[
\omega_{\mathbf A}^{\textsf x}(r):=\sup_{C \in \mathbb{C}_r}\fint_{C} \,\Abs{\mathbf A(t,x)- \bar{\mathbf A}^{\textsf x}_{C}(t)}\,dxdt.
\]

\begin{condition}			\label{cond1}
The principal coefficient matrix $\mathbf A=(a^{ij})$ is symmetric and there is a constant $\delta \in (0,1]$ such that the eigenvalues of $\mathbf A(t,x)$ are in $[\delta, 1/\delta]$ for all $(t,x) \in \mathbb{R}^{d+1}$.
Moreover, $\mathbf{A}$ is of Dini mean oscillation in $x$, that is,
\[
\int_0^{1} \frac{\omega_{\mathbf A}^{\textsf x}(t)}{t}\,dt <\infty.
\]
\end{condition}

We shall denote $\mathbf A \in \mathrm{DMO}_x$ if $\mathbf A$ is of Dini mean oscillation in $x$.
We observe that $\mathbf A \in \mathrm{VMO}_x$ if $\mathbf A \in \mathrm{DMO}_x$; refer to \cite{Krylov2007} for definition of $\mathrm{VMO}_x$.

\begin{condition}			\label{cond2}
Let  $p$, $q$, and $\beta$ be such that
\[
p,q \in (1,\infty),\quad \beta \in (0,1),\quad d/p+2/q \ge \beta,
\]
and denote
\[
\tilde p=\beta p /(\beta+1),\quad \tilde q=\beta q /(\beta+1).
\]
The lower-order coefficients $\vec b=(b^1,\ldots, b^d)$ and $c$ are such that $\vec b \in E_{p,q,\beta}$ and $c \in E_{\tilde p,\tilde q,\beta+1}$ with estimates
\[
\norm{\vec b}_{E_{p,q,\beta}}<\mathfrak{b},\quad \norm{c}_{E_{\tilde p,\tilde q,\beta+1}}<\mathfrak{c},
\]
for some $\mathfrak{b}$, $\mathfrak{c}<\infty$.
\end{condition}

\begin{remark}
Let $p$, $q \in (1,\infty)$ be such that $d/p+2/q<1$ and let $\beta=d/p+2/q$.
Clearly, for $C \in \mathbb{C}_\rho$, we have
\[
\rho^\beta \dashnorm{g}_{L_{p,q}(C)} \le N(d) \norm{g}_{L_{p,q}(\bR^{d+1})}.
\]
Moreover, it follows from H\"older's inequality that for $\rho \le 1$, we have
\[
\rho^{1+\beta} \dashnorm{g}_{L_{\tilde p,\tilde q}(C)} \le \rho^{2\beta} \dashnorm{g}_{L_{\tilde p,\tilde q}(C)} \le \rho^{2\beta} \dashnorm{g}_{L_{p/2,q/2}(C)} \le N(d) \norm{g}_{L_{p/2,q/2}(\bR^{d+1})}.
\]
Therefore, if $\vec b \in L_{p,q}$ and $c \in L_{p/2,q/2}$, then Condition \ref{cond2} is satisfied with $\beta=d/p+2/q$.
\end{remark}

\begin{definition}
Let $Q \subset \bR^{d+1}$ be a domain.
We say that $u \in L_{1,\mathrm{loc}}(Q)$ is a weak solution of
\[
\mathscr{L}^*u -c u =0 \quad\text{in }\; Q
\]
if $\vec b u$, $cu \in L_{1,\mathrm{loc}}(Q)$, and the following identity holds for any $\eta \in C^{\infty}_c(Q)$:
\begin{align}			
						\nonumber
I&:=\iint_{\bR^{d+1}} u \left(\mathscr{L} \eta -c\eta\right) dxdt\\
						\label{eq2107mon}
&=\iint_{\bR^{d+1}} u \partial_t \eta  + a^{ij} u D_{ij}\eta+ b^i u D_i\eta - cu\eta \,dxdt =0.
\end{align}
\end{definition}

The following lemma is an easy consequence of \cite[Theorem 3.2]{DEK21} by following essentially the same the arguments used in \cite{Brezis}; see also \cite{DK17}.
\begin{lemma}			\label{lem1916sun}
Assume Condition \ref{cond1}.
Let $C_r= (-r^2,0]\times B_r \subset \bR^{d+1}$ and let $u \in L_{1}(C_{2r})$ be a solution to $-\partial_t u+ \dv^2(\mathbf A u)=f$ in $C_{2r}$, where $f \in L_p(C_{2r})$ for some $p \in (1,\infty)$, then $u \in L_{p}(C_{r})$ and we have
\[
\norm{u}_{L_p(C_r)} \le N \left(\norm{u}_{L_1(C_{2r})}+ \norm{f}_{L_p(C_{2r})}\right),
\]
where $N$ depends only on $d$, $\delta$, $p$, $r$, and $\omega_{\mathbf A}^{\textsf x}$.
\end{lemma}
\begin{proof}
We can assume $r=1$.
Also, it is enough to show that $u \in L_{q}(C_1)$ for some $q>1$; see \cite{EM2017}.
By definition of a weak solution, we have
\begin{equation}				\label{eq1602wed}
\iint_{C_2} u \partial_t \eta +a^{ij}u D_{ij} \eta =\iint_{C_2} f \eta,\quad \forall \eta \in C^{\infty}_c(C_2).
\end{equation}
For any $g \in C^{\infty}_c(C_2)$, let $v \in W^{1,2}_2(C_2)$ be the solution of the problem
\begin{equation}				\label{eq1603wed}
v_t +a^{ij}D_{ij} v = g \;\text{ in }\;C_2
\end{equation}
with zero boundary condition on the parabolic boundary of $C_2$.
By \cite[Theorems 1.2]{DEK21} (see also \cite[Theorem 3.2]{DEK21}), we have $v \in W^{1,2}_\infty(C_{3/2})$.
Let $v^{(\varepsilon)}$ be the standard mollification of $v$.
For any fixed cut-off function $\zeta \in C^\infty_c(C_{3/2})$ such that $\zeta=1$ on $C_1$, we take $\eta=\zeta v^{(\varepsilon)} \in C^{\infty}_c(C_2)$.
Then by \eqref{eq1602wed}, we have
\begin{equation}				\label{eq1450thu}
\iint_{C_2} u \left(\zeta_t v^{(\varepsilon)} + \zeta v^{(\varepsilon)}_t\right) +a^{ij}u \left(D_{ij} \zeta v^{(\varepsilon)} + 2 D_i\zeta D_j v^{(\varepsilon)} + \zeta D_{ij}v^{(\varepsilon)}\right) =\int_{C_2} f \zeta v^{(\varepsilon)}.
\end{equation}
There is a sequence $\varepsilon_n$ converging to $0$ such that
\[
v^{(\varepsilon_n)} \to v, \quad D v^{(\varepsilon_n)} \to Dv,\quad D^2 v^{(\varepsilon_n)} \to D^2 v, \quad v^{(\varepsilon_n)}_t \to v_t\quad\text{a.e.}
\]
Also, note that for all sufficiently small $\varepsilon>0$ we have
\[
\norm{v^{(\varepsilon)}}_{W^{1,2}_\infty(\supt \zeta)} \le \norm{v}_{W^{1,2}_\infty(C_{3/2})}<\infty.
\]
Hence, by using the dominated convergence theorem, we find that \eqref{eq1450thu} is valid with $v^{(\varepsilon)}$ replaced by $v$.
On the other hand, by multiplying \eqref{eq1603wed} with $\zeta u$ and noting that $v_t$, $D^2 v$, $g \in L_\infty(\supt \zeta)$ and $\zeta u \in L_1(C_2)$,  we have
\begin{equation}			\label{eq1627thu}
\iint_{C_2} v_t \zeta u + a^{ij} D_{ij}v \zeta u  = \iint_{C_2} g \zeta u.
\end{equation}
By combing \eqref{eq1627thu} and \eqref{eq1450thu} with $v$ in place of $v^{(\epsilon)}$, we obtain
\[
\iint_{C_2} \zeta u g = \iint_{C_2} \zeta_t u v + a^{ij} D_{ij}\zeta uv+ 2a^{ij}u D_i\zeta D_jv+ \iint_{C_2} f \zeta v.
\]
Note that we can choose $\zeta$ such that
\[
\norm{\zeta}_\infty+ \norm{\zeta_t}_\infty+ \norm{D^2 \zeta}_\infty + \norm{D\zeta}_\infty \le 64.
\]
Therefore, we have
\begin{equation}			\label{eq1650thu}
\Abs{\iint_{C_2} \zeta u g\,dxdt} \le N\left( \norm{v}_{L_\infty(C_2)} +\norm{Dv}_{L_\infty(C_2)}\right) \norm{u}_{L_1(C_2)} + N \norm{f}_{L_p(C_2)} \norm{v}_{L_\infty(C_2)}.
\end{equation}
Since $a^{ij} \in \mathrm{DMO}_x \subset \mathrm{VMO}_x$, we have for $1<q<\infty$ that
\[
\norm{v}_{W^{1,2}_q(C_2)} \le N \norm{g}_{L_q(C_2)}.
\]
Then by the Sobolev embedding (see, for instance, \cite{LSU}, and also Remark~\ref{rmk1909sun}), we have for $q>d+2$ that
\[
\norm{v}_{L_\infty(C_2)}+ \norm{Dv}_{L_\infty(C_2)} \le N \norm{g}_{L_q(C_2)}.
\]
Since $g \in C^\infty_c(C_2)$ is arbitrary, by combining the previous inequality with \eqref{eq1650thu} and using the converse of H\"older's inequality, we have
\[
\norm{u}_{L_{q'}(C_1)}\le \norm{\zeta u}_{L_{q'}(C_2)}  \le N \left( \norm{u}_{L_1(C_2)}+ \norm{f}_{L_p(C_2)} \right),
\]
where $1<q' <\frac{d+2}{d+1}$.
\end{proof}

The following lemma complements \cite[Lemma 2.8]{Krylov2023}.
\begin{lemma}			\label{lem16.48sat}
Let $0<\beta \le d/p+2/q$ and $\beta<1$.
Then for any $u \in E^{1,2}_{p,q,\beta}(\bR^{d+1})$ and $C \in \mathbb{C}_1$, we have
\[
[u]_{\mathscr{C}^{(2-\beta)/2,2-\beta}(C)} \le \left(\norm{D^2u}_{E_{p,q,\beta}}+\norm{\partial_t u}_{E_{p,q,\beta}}\right).
\]
\end{lemma}

\begin{proof}
Using \cite[Lemma 3.2]{Krylov2007} and H\"older's inequality, for any $u \in C^\infty_c(\bR^{d+1})$, we have
\begin{equation*}				
\int_{\tilde C_r} \abs{D_iu-(D_iu)_{\tilde C_r}} \le N r^{1+d+2-d/p-2/q} \left(\norm{D^2u}_{L_{1}(\tilde C_r)}+\norm{\partial_t u}_{L_{1}(\tilde C_r)}\right).
\end{equation*}
Then by \eqref{eq_morrey}, for $r \le 1$, we have
\[
\int_{\tilde C_r} \abs{D_iu-(D_iu)_{\tilde C_r}} \le Nr^{d+2+1-\beta} \left(\norm{D^2u}_{E_{p,q,\beta}}+\norm{\partial_t u}_{E_{p,q,\beta}}\right).
\]
Then, applying Campanato's theorem, for any $C \in \mathbb{C}_1$, we obtain
\begin{equation}				\label{eq0613thu}
[Du]_{\mathscr{C}^{(1-\beta)/2,1-\beta}(C)}\le N \left(\norm{D^2u}_{E_{p,q,\beta}}+\norm{\partial_t u}_{E_{p,q,\beta}}\right).
\end{equation}

Next, we aim to demonstrate that for any $C \in \mathbb{C}_1$, we have
\begin{equation}				\label{eq0750thu}
\sup_{\substack{(t,x), (s,x)\in C\\t \neq s}} \frac{\abs{u(t,x)-u(s,x)}}{\abs{t-s}^{(2-\beta)/2}} \le N \left(\norm{D^2u}_{E_{p,q,\beta}}+\norm{\partial_t u}_{E_{p,q,\beta}}\right).
\end{equation}
We consider a non-negative smooth function $\eta$ with compact support in $B_1(0) \subset \mathbb{R}^d$, satisfying $\int_{\mathbb{R}^d} \eta = 1$. Additionally, we assume that $\eta$ is a radial function.
The spatial mollification $u^{(\epsilon)}$ of $u$ is defined as follows:
\[
u^{(\epsilon)}(t,x)=\int_{\bR^d} u(t, x-\epsilon y) \eta(y)\,dy=\frac{1}{\epsilon^d}\int_{\bR^d} u(t,y)\eta\left(\frac{x-y}{\epsilon}\right)dy.
\]
For $(t,x) \in \bR^{d+1}$, we shall estimate $\abs{u(t+r^2,x)-u(t,x)}$ as follows:
\begin{multline}			\label{eq0728thu}
\abs{u(t+r^2,x)-u(t,x)} 
\le  \abs{u(t+r^2,x)-u^{(r)}(t+r^2,x)}+\abs{u(t,x)-u^{(r)}(t,x)}\\
 +\abs{u^{(r)}(t+r^2,x)-u^{(r)}(t,x)}.
\end{multline}
By using the assumptions that $\eta(y)=\eta(-y)$ and $\int_{B_1}\eta =1$, for any $s \in \bR$, we obtain
\[
u^{(r)}(s,x)-u(s,x)  = \frac12\int_{B_1}  \bigl\{ u(s,x+ry)+u(s,x-ry) -2u(s,x)\bigr\} \,\eta(y)\,dy.
\]
Since
\begin{align*}
u(s,x+ry)+u(s,x-ry) - 2u(s,x) &= \int_0^r \frac{d}{d\tau} u(s, x+\tau y)\,d\tau+ \int_0^r \frac{d}{d\tau} u(s, x-\tau y)\,d\tau \\
&= \int_0^r \left(D_iu(s,x+\tau y)-D_iu(s,x-\tau y) \right)y_i\,d\tau,
\end{align*}
we have
\begin{align}		
				\nonumber
\Abs{u^{(r)}(s,x)-u(s,x)}&=\frac12 \Abs{\int_{B_1}\int_0^r  \left\{D_iu(s,x+\tau y)-D_iu(s,x-\tau y) \right\}y_i \eta(y)\,d\tau  dy}\\
				\nonumber
&\le \frac12 \norm{\eta}_{\infty}\int_{B_1}\int_0^r  \Abs{Du(s,x+\tau y)-Du(s,x-\tau y)} d\tau  dy\\
				\label{eq1526bt}
&\le N r^{2-\beta} \left(\norm{D^2u}_{E_{p,q,\beta}}+\norm{\partial_t u}_{E_{p,q,\beta}}\right),
\end{align}
where we have used \eqref{eq0613thu} in the last inequality.

On the other hand, we have
\begin{align*}
u^{(r)}(t+r^2,x)-u^{(r)}(t,x) &= \int^{t+r^2}_{t} \partial_s u^{(r)}(s,x)\,ds\\		
&= \int^{t+r^2}_{t} \int_{B_1} \partial_s u(s,x-ry)\eta(y)\,dy ds.
\end{align*}
Therefore, 
\begin{align}
			\nonumber
\Abs{u^{(r)}(t+r^2,x)-u^{(r)}(t,x)} &\le  \norm{\eta}_{\infty} r^{-d} \int_{\tilde C_r(t,x)} \abs{\partial_s u(s,y)}dyds\\
			\label{eq1607bt}
&\le N r^{2-d/p-2/q}\norm{\partial_t u}_{L_{p,q}(\tilde C_r)} \le N r^{2-\beta}\norm{\partial_t u}_{E_{p,q,\beta}}.
\end{align}
Combining \eqref{eq0728thu}, \eqref{eq1526bt}, and \eqref{eq1607bt}, we have
\[
\abs{u(x,t+r^2)-u(x,t)} \le Nr^{2-\beta}\left(\norm{D^2u}_{E_{p,q,\beta})}+\norm{\partial_t u}_{E_{p,q,\beta}}\right).
\]
We have proved \eqref{eq0750thu}.
\end{proof}

\begin{remark}			\label{rmk1909sun}
The proof of Lemma~\ref{lem16.48sat} shows the following well known result: if $u \in W^{1,2}_{p,q}$ with $d/p + 2/q < 1$, then 
\[
[u]_{\mathscr{C}^{(1+\alpha)/2,1+\alpha}} \le N \left(\norm{D^2u}_{L_{p,q}}+\norm{\partial_t u}_{L_{p,q}}\right).
\]
However, there is no restrictions on $d/p+2/q$ in the previous lemma.
\end{remark}
The following result can be inferred from \cite{DEK21}, although it was not explicitly stated there.
For the reader's convenience, we provide the statement here.

\begin{theorem}			\label{thm:harnack}
Assume Condition \ref{cond1} holds.
Let $R>0$ be a fixed number, $0<r<R/2$, and $(t_0, x_0) \in \bR^{d+1}$.
Suppose $u \in L_1(C_{2r})$ is a nonnegative weak solution of
\[
-\partial_tu +\dv^2(\mathbf{A} u) = 0 \;\text{ in }\;C_{2r}.
\]
Then, we have
\[
\sup_{(t_0-3r^2, t_0-2r^2)\times B_{r}(x_0)} u \le N \inf_{(t_0-r^2, t_0)\times B_{r}(x_0)} u,
\]
where $N$ is a constant depending only on $d$, $\delta$, $\omega_{\mathbf A}^{\textsf x}$, and $R$.
\end{theorem}
\begin{proof}
By Lemma~\ref{lem1916sun}, we have $u \in L_{p,{\rm loc}}(C_{2r})$, for any $1<p<\infty$, and thus by \cite[Theorem 1.4]{DEK21}, it is continuous in $C_{2r}$.
The conclusion of the lemma is derived by consolidating Lemmas 5.2 and 5.7 in \cite{DEK21}.
\end{proof}

\section{Proof of the main theorem}				\label{sec3}

\subsection{Regularization of singular zeroth order coefficient}	\label{sec3.1}
We shall first demonstrate how to transform a singular $c$ to a regular $\tilde c$.
Let $u$ be a solution to
\[
\mathscr{L}^* u -cu=0 \quad \text{in }Q\subset \bR^{d+1},
\]
Then we have the identity
\begin{equation}			\label{eq0918tue}
I=\iint_{\bR^{d+1}} u (\mathscr{L}-c)\eta\, dxdt=0,\quad \forall \eta \in C^\infty_c(Q).
\end{equation}
Let $\zeta$ be a function to be determined.
By taking $\zeta \eta$ in place of $\eta$ in \eqref{eq0918tue}, we have
\begin{equation}			\label{eq1034tue}
I=\iint_{\bR^{d+1}} u(\mathscr{L}-c)\zeta \eta\,dxdt=0,\quad \forall \eta \in C^\infty_c(Q).
\end{equation}
We note that
\begin{align}
				\nonumber
(\mathscr{L}-c)\zeta \eta &= \zeta \mathscr{L}\eta + \eta \mathscr{L}\zeta + 2a^{ij}D_i \eta D_j \zeta - c \zeta \eta\\
&=\zeta \left(\mathscr{L} + 2a^{ij}\frac{D_i \zeta}{\zeta} D_j \right)\eta + \eta (\mathscr{L}-c)\zeta.		\label{eq0944tue}
\end{align}
We shall consider $\zeta$ that is of the form $\zeta=1+\zeta_\lambda$, where $\zeta_\lambda$ is a solution to the problem
\begin{equation}			\label{eq2029sat}
(\mathscr{L}-c -\lambda ) \zeta_\lambda = c \quad\text{in}\quad \bR^{d+1}.
\end{equation}
We note that for $r \le 1$, we have
\[
r \sup_{C \in \mathbb{C}_r} \dashnorm{\vec b}_{L_{(\beta+1)\tilde p, (\beta+1)\tilde q}(C)}  \le r \sup_{C \in \mathbb{C}_r}\dashnorm{\vec b}_{L_{p,q}(C)} \le r^{1-\beta} \sup_{C \in \mathbb{C}_r} r^\beta \dashnorm{\vec b}_{L_{p,q}(C)} \le r^{1-\beta} \mathfrak{b}.
\]
Therefore, for any $\check{b} \in (0,1]$, there exists a $\rho \in (0,1]$ such that we have 
\[
\sup_{r \le \rho} r\, \sup_{C \in \mathbb{C}_r} \dashnorm{\vec b}_{L_{(\beta+1)\tilde p, (\beta+1)\tilde q}(C)} \le \rho^{1-\beta} \mathfrak{b} \le \check{b}.
\]

By \cite[Theorem 3.5]{Krylov2023}, there exists a constant $\lambda_0\ge 1$ depending only on $d$, $\delta$, $\omega_{\mathbf A}^{\mathsf x}$, $p$, $q$, $\beta$, $\mathfrak{b}$, and $\mathfrak{c}$ such that for $\lambda \ge \lambda_0$, the equation \eqref{eq2029sat} has a unique solution $\zeta_\lambda$ in the function space $E^{1,2}_{\tilde p, \tilde q,\beta+1}$ satisfying
\begin{equation}			\label{eq1815sat}
\norm{\partial_t \zeta_\lambda,\,D^2 \zeta_\lambda \, \sqrt{\lambda}D\zeta_\lambda,\,\lambda \zeta_\lambda}_{E_{\tilde p,\tilde q,\beta+1}} \le N,
\end{equation}
where $N$ is a constant depending only on $d$, $\delta$, $\omega_{\mathbf A}^{\mathsf x}$, $p$, $q$, $\beta$, $\mathfrak{b}$, and $\mathfrak{c}$.

\begin{lemma}			\label{lem2019thu}
There exists a constant $N$ that depends only on $d$, $\delta$, $\omega_{\mathbf A}^{\mathsf x}$, $p$, $q$, $\beta$, $\mathfrak{b}$, and $\mathfrak{c}$, such that the following estimates hold:
\begin{equation}			\label{eq2141sat}
\norm{\zeta_\lambda}_{L_\infty(\bR^{d+1})} \le N \lambda^{(\beta-1)/2},\quad[\zeta_\lambda]_{\mathscr{C}^{(1-\beta)/2, 1-\beta}(C)} \le N,\;\; \forall C \in \mathbb{C}_1,\quad \norm{D\zeta_\lambda}_{E_{p,q,\beta}} \le N.
\end{equation}
\end{lemma}
\begin{proof}
According to \cite[Lemma 2.5]{Krylov2023}, for any $\epsilon \in (0,1]$, we have
\[
\sup_{\bR^{d+1}}\, \abs{\zeta_\lambda} \le \epsilon^{1-\beta}\left(\norm{\partial_t \zeta_\lambda}_{E_{\tilde p,\tilde q,\beta+1}}+\norm{D^2 \zeta_\lambda}_{E_{\tilde p,\tilde q,\beta+1}} \right)+N \epsilon^{-\beta-1} \norm{\zeta_\lambda}_{E_{\tilde p,\tilde q,\beta+1}}.
\]
Then, from \eqref{eq1815sat}, for $\epsilon \in (0,1]$ and $\lambda \ge \lambda_0$, we have
\[
\abs{\zeta_\lambda} \le N \left(\epsilon^{1-\beta} + \epsilon^{-\beta-1}\lambda^{-1}\right).
\]
By taking $\epsilon=(\lambda_0/\lambda)^{1/2}$ in the above, we get the first part of \eqref{eq2141sat}.
The second and third part of \eqref{eq2141sat} follow from \eqref{eq1815sat} and Lemmas 2.6 and 2.8 in  \cite{Krylov2023}.
\end{proof}

By Lemma~\ref{lem2019thu}, we can choose $\lambda$ large enough such that $\abs{\zeta_\lambda} \le \frac12$.
With this choice of $\lambda$, we have $\zeta=1+\zeta_\lambda$ satisfies
\begin{equation}			\label{eq1934sat}
1/2 \le \zeta \le 2.
\end{equation}
Moreover, for $X=(t,x)$ and $Y=(s,y)$ with $\abs{X-Y}_p <1$, we have
\[
\frac{\abs{\zeta(X)-\zeta(Y)}}{\;\;\abs{X-Y}_p^{1-\beta}} \le \sup_{C \in \mathbb C_1} \,[\zeta_\lambda]_{\mathscr{C}^{(1-\beta)/2,1-\beta}(C)} \le N.
\]
Note that the left-hand side of the previous inequality is bounded by $2 \sup_{\mathbb R^{d+1}} \abs{\zeta}$ whenever $\abs{X-Y}_p \ge 1$.
Hence, we deduce that
\begin{equation}			\label{eq0825sun}
[\zeta]_{\mathscr{C}^{(1-\beta)/2,1-\beta}(\mathbb R^{d+1})} \le N.
\end{equation}

Since we have
\[
(\mathscr{L} -c)\zeta= (\mathscr{L} -c)(1+\zeta_\lambda)= -c + (\mathscr{L} -c-\lambda) \zeta_\lambda+ \lambda \zeta_\lambda=\lambda \zeta_\lambda=\lambda(\zeta-1),
\]
from \eqref{eq0944tue}, we obtain
\begin{equation}			\label{eq1042tue}
(\mathscr{L}-c)\zeta \eta= \zeta \left(\mathscr{L} + 2a^{ij}\frac{D_i \zeta}{\zeta} D_j +\lambda \frac{\zeta-1}{\zeta} \right)\eta.
\end{equation}
We define
\begin{equation}			\label{eq1948sat}
\bar b^j= b^j + 2a^{ij}\frac{D_i \zeta}{\zeta},\quad \bar c=- \lambda\left(\frac{\zeta-1}{\zeta}\right).
\end{equation}
Then from \eqref{eq1034tue} and \eqref{eq1042tue}, we find that $\bar u= \zeta u$ satisfies
\begin{equation}			\label{eq2033mon}
I=\iint_{\bR^{d+1}} \bar u \left(a^{ij}D_{ij} + \bar b^i D_i -\bar c \right) \eta =0,
\end{equation}
and thus $\bar u$ is a solution to
\[
-\partial_t \bar u + D_{ij}(a^{ij}u)-D_i(\bar b^i u) -\bar c u=0\;\text{ in }\;Q.
\]

\begin{remark}			\label{rmk_c}
We observe from \eqref{eq1934sat} and \eqref{eq1948sat} that the functions $\bar b^j - b^j$ are bounded, implying that $\bar b^j$ belongs to the same $E_{p,q, \beta}$ space as $b^j$.
Additionally, it follows from \eqref{eq1934sat}, \eqref{eq0825sun}, and \eqref{eq1948sat} that $\bar c$ is a bounded H\"older continuous function.
More quantitatively, we have
\begin{gather*}
\norm{\bar{\vec b}}_{E_{p,q,\beta}} \le \norm{\vec b}_{E_{p,q,\beta}}+4\delta^{-1}\norm{D\zeta}_{E_{p,q,\beta}} \le N,\\
\norm{\bar c}_{L_\infty(\mathbb{R}^{d+1})} \le \lambda\norm{(\zeta-1)/\zeta}_{L_\infty} \le N \lambda \le N,\\
[\bar c]_{\mathscr{C}^{(1-\beta)/2,1-\beta}(\mathbb{R}^{d+1})} \le  \lambda [\zeta]_{\mathscr{C}^{(1-\beta)/2,1-\beta}(\mathbb{R}^{d+1})}/\norm{\zeta}_{L_\infty(\mathbb{R}^{d+1})}^2 \le N \lambda \le N,
\end{gather*}
where $N$ depends only on $d$, $\delta$, $\omega_{\mathbf A}^{\mathsf x}$, $p$, $q$, $\beta$, $\mathfrak{b}$, and $\mathfrak{c}$.
\end{remark}

\subsection{Regularization of singular drift term}				\label{sec3.2}
We proceed to regularize the drift term through Zvonkin's transform, which was first introduced in \cite{Zvonkin}. 
Below we use a modification of this transformation introduced in \cite{FGP2010}.

Let $\vec \Phi: \bR^{d+1} \to \bR^{d+1}$ be an invertible mapping defined as
\[
\vec\Phi(t,x)=(t, \vec \phi(t,x)),
\]
where $\vec \phi=(\phi^1,\ldots, \phi^d)$.
The inverse of $\vec \Phi$, denoted by $\vec \Psi$, is given by
\[
\vec \Psi(t,y)=(t, \vec \psi(t,y)).
\]
It is important to note that for each fixed $t$, we have $\vec\psi(t,\cdot)$ as the inverse of $\vec \phi(t,\cdot)$, meaning
\begin{equation*}			
\vec \psi(t,\vec\phi(t,x))=x,\quad \vec\phi(t,\vec \psi(t,y))=y.
\end{equation*}
By utilizing $\eta \circ \vec \Phi$ as a test function in place of $\eta$ in equation \eqref{eq2107mon}, we have
\begin{equation}			\label{eq2207mon}
I=\iint_{\bR^{d+1}} u (\mathscr{L} -c) \eta \circ \vec\Phi\, dxdt=0.
\end{equation}
Clearly, we can observe that
\begin{align*}
\partial_t (\eta \circ \vec \Phi)(t,x) &= \partial_t \eta(t, \vec \phi(t,x)) +D_k \eta(t, \vec\phi(t,x))  \phi^k_t(t,x),\\
D_i (\eta \circ \vec \Phi)(t,x) &= D_k\eta(t, \vec \phi(t,x)) D_i\phi^k(t,x),\\
D_{ij} (\eta \circ \vec \Phi)(t,x) &= D_{kl}\eta(t, \vec \phi(t,x)) D_i\phi^k(t,x) D_j\phi^l(t,x) +D_k \eta(t, \vec\phi(t,x))  D_{ij}\phi^k(t,x),
\end{align*}
Therefore, we can express
\begin{multline}				\label{eq2135mon}
(\mathscr{L}-c) \eta(\vec\Phi (t,x))= \partial_t \eta(t, \vec\phi(t,x)) + \tilde a^{kl}(t,x) D_{kl}\eta(t,\vec\phi(t,x))-c(t,x) \eta(t,\vec\phi(t,x))\\
+ D_k \eta(t,\vec \phi(t,x)) \mathscr{L} \phi^k(t,x),
\end{multline}
with the notation
\begin{equation}			\label{eq2137mon}
\tilde a^{kl}(t,x)= a^{ij}(t,x) D_i\phi^k(t,x) D_j \phi^l(t,x)
\end{equation}

We will consider $\vec \phi(t,x)$ in the form $\vec \phi(t,x)=\vec \phi_\lambda(t,x)=x+\vec \xi_\lambda(t,x)$, where $\vec \xi_\lambda$ satisfies the system of equations
\begin{equation}				\label{eq1652fri}
(\mathscr{L}-\lambda) \vec \xi_\lambda=-\vec b\;\text{ in }\;\mathbb{R}^{d+1}.
\end{equation}

Once again, according to \cite[Theorem 3.5]{Krylov2023}, there exist a constant $\lambda_0>0$ depending only on $d$, $\delta$, $\omega_{\mathbf A}^{\mathsf x}$, $p$, $q$, $\beta$, $\mathfrak{b}$, and $\mathfrak{c}$, such that for $\lambda \ge \lambda_0$, the system \eqref{eq1652fri} has a unique solution $\vec \xi_\lambda=(\xi_\lambda^1,\ldots, \xi_\lambda^d)$.

\begin{lemma}			\label{lem2036sat}
There exists a constant $N$ that depends only on $d$, $\delta$, $\omega_{\mathbf A}^{\mathsf x}$, $p$, $q$, $\beta$, $\mathfrak{b}$, and $\mathfrak{c}$, such that the following estimates hold:
\begin{equation}				\label{eq2037sat}
\sup_{\bR^{d+1}}\,\abs{\vec \xi_\lambda} \le N \lambda^{(\beta-2)/2},\quad
\sup_{\bR^{d+1}}\,\abs{D \vec \xi_\lambda} \le N \lambda^{(\beta-1)/2}.
\end{equation}
\end{lemma}

\begin{proof}
Similar to the proof of Lemma~\ref{lem2019thu}, by \cite[Theorem 3.5]{Krylov2023},  for $\lambda \ge \lambda_0$, we have
\begin{equation}			\label{eq1000sun}
\norm{\partial_t \vec \xi_\lambda,\, D^2 \vec \xi_\lambda,\,\sqrt{\lambda}D \vec \xi_\lambda, \,\lambda \vec \xi_\lambda}_{E_{p,q,\beta}} \le N \norm{\vec b}_{E_{p,q,\beta}} \le N.
\end{equation}
Then, utilizing \cite[Lemma 2.5]{Krylov2023}, for $\varepsilon \in (0,1]$, the estimates
\begin{align*}
\sup_{\bR^{d+1}}\, \abs{\vec \xi_\lambda} \le N \left(\varepsilon^{2-\beta} + \varepsilon^{-\beta}\lambda^{-1}\right),\quad 
\sup_{\bR^{d+1}}\,\abs{D\vec \xi_\lambda} \le N \left(\varepsilon^{1-\beta} + \varepsilon^{-\beta-1}\lambda^{-1}\right)
\end{align*}
holds, where the constant $N$ depends solely on $d$, $\delta$, $\omega_{\mathbf A}^{\mathsf x}$, $p$, $q$, $\beta$, $\mathfrak{b}$, and $\mathfrak{c}$.
By setting $\epsilon=(\lambda_0/\lambda)^{1/2}$, we obtain \eqref{eq2037sat}.
\end{proof}

\begin{lemma}			\label{lip_one_half}
For $s$, $t\in \bR$, such that $\abs{s-t} \le 1$ and $x \in \bR^d$, we have
\[
\abs{\vec \xi_\lambda(t,x)-\vec \xi_\lambda(t,x)} \le N \lambda^{-(1-\beta)/2} \abs{t-s}^{1/2}.
\]
\end{lemma}
\begin{proof}
By Lemma~\ref{lem16.48sat} and \eqref{eq1000sun}, for any $C \in \mathbb{C}_1$, we have
\begin{equation}			\label{eq1210sun}
\sup_{\substack{(t,x), (s,x)\in C\\t \neq s}} \frac{\abs{\vec \xi_\lambda(t,x)-\vec \xi_\lambda(s,x)}}{\abs{t-s}^{(2-\beta)/2}} \le N \left(\norm{D^2 \vec \xi_\lambda}_{E_{p,q,\beta}}+\norm{\partial_t \vec \xi_\lambda}_{E_{p,q,\beta}}\right) \le N.
\end{equation}
On the other hand, by \eqref{eq2037sat}, we have 
\begin{equation}			\label{eq1211sun}
\sup_{(t,x)\in C} \,\abs{\vec \xi_\lambda(t,x)} \le N \lambda^{(\beta-2)/2}.
\end{equation}
The lemma is established through interpolation inequalities, as demonstrated in, for example, \cite[Theorem 3.2.1]{Krylov1996}.
To ensure completeness, we present the argument below.
Denote 
\[
A:=\sup_{\substack{(t,x), (s,x)\in C\\t \neq s}} \frac{\abs{\vec \xi_\lambda(t,x)-\vec \xi_\lambda(s,x)}}{\abs{t-s}^{1/2}}\quad\text{and}\quad B:=\sup_{\substack{(t,x), (s,x)\in C\\t \neq s}} \frac{\abs{\vec \xi_\lambda(t,x)-\vec \xi_\lambda(s,x)}}{\abs{t-s}^{(2-\beta)/2}}.
\]
For any $\epsilon>0$, we claim that
\begin{equation}			\label{eq1150sun}
A \le \epsilon B + N(\beta) \epsilon^{-1/(1-\beta)}\, \sup_{C}\, \abs{\vec \xi_\lambda}.
\end{equation}
We can choose $(t,x)$, $(s,x) \in C$ be such that
\[
\frac12 A \le \frac{\abs{\vec \xi_\lambda(t,x)-\vec \xi_\lambda(s,x)}}{\abs{t-s}^{1/2}} \le \abs{t-s}^{(1-\beta)/2} B.
\]
If $\abs{t-s} \le (\epsilon/2)^{2/(1-\beta)}$, we have $A \le \epsilon B$, and \eqref{eq1150sun} is true.
Otherwise, we have
\[
A \le 4 \abs{t-s}^{-1/2}\,\sup_{C}\, \abs{\vec \xi_\lambda} \le 4(\epsilon/2)^{-1/(1-\beta)}\, \sup_{C}\, \abs{\vec \xi_\lambda}.
\]
We have proved the assertion \eqref{eq1150sun}.
Now, the lemma follows by taking $\epsilon=\lambda^{-(1-\beta)/2}$ in \eqref{eq1150sun}, and utilizing \eqref{eq1210sun} and \eqref{eq1211sun}.
\end{proof}

Using Lemmas \ref{lem2036sat} and \ref{lip_one_half}, for any $\varepsilon \in (0,\frac12)$, there exists $\lambda$ sufficiently large such that any $(t,x) \in \bR^d$, we have
\begin{equation}			\label{eq1255mon}
\abs{D \vec \xi_\lambda(t,x)} \le \varepsilon,
\end{equation}
and for any $s$, $t \in \mathbb{R}$, with $\abs{s-t} \le 1$, and $x\in \bR^d$, we have
\[
\abs{\vec \xi_\lambda(t,x)-\vec \xi_\lambda(s,x)} \le \varepsilon \abs{t-s}^{1/2}.
\]
Then, with this choice of $\lambda$, the function $\vec \phi_\lambda(t,x)=x+\vec \xi_\lambda(t,x)$ satisfies
\begin{equation}			\label{eq1309sun}
(1-\varepsilon)\abs{x-y} \le \abs{\vec \phi_\lambda(t,x)-\vec \phi_\lambda(t,y)} \le (1+\varepsilon) \abs{x-y},\quad \forall (t,x)\in \bR^{d+1},
\end{equation}
and for any $s$, $t \in \mathbb{R}$, with $\abs{s-t} \le 1$, the following holds:
\begin{equation}			\label{eq1359sat}
\sup_{x \in \bR^d}\, \abs{\vec \phi_\lambda(t,x)-\vec \phi_\lambda(s,x)} \le \varepsilon \abs{t-s}^{1/2}.
\end{equation}
Since $D\vec \phi_\lambda(t,x)=I+D\vec \xi_\lambda(t,x)$, we may assume that $\lambda$ is chosen such that for each $(t,x) \in \bR^{d+1}$, we have
\begin{equation}			\label{eq2118sat}
\frac12 \le \det(D \vec\phi_\lambda(t,x)) \le 2.
\end{equation}

It is evident that for each $t\in \bR$, the mapping $\vec \phi_\lambda(t,\cdot)$ is a $C^1$ diffeomorphism on $\bR^d$.
Denote by $\vec \psi_\lambda(t,\cdot)$ the inverse of $\vec \phi_\lambda(t,\cdot)$ on $\bR^d$.
Note that
\[
\vec \phi_\lambda(t, \vec \psi_\lambda(t,x))-\vec \phi_\lambda(s, \vec \psi_\lambda(t,x))=\vec \phi_\lambda(s, \vec \psi_\lambda(s,x))-\vec \phi_\lambda(s, \vec \psi(t,x))
\]
and by \eqref{eq1359sat} and \eqref{eq1309sun} that
\[
\abs{{\rm LHS}} \le \varepsilon \abs{t-s}^{1/2},\quad \abs{{\rm RHS}} \ge (1-\varepsilon) \abs{\vec \psi_\lambda(s,x)-\vec \psi_\lambda(t,x)}.
\]
Therefore, for any $s$, $t \in \mathbb{R}$, with $\abs{s-t} \le 1$, the following holds:
\begin{equation}			\label{eq1448mon}
\abs{\vec \psi_\lambda(t,x)-\vec \psi_\lambda(s,x)} \le \varepsilon(1-\varepsilon)^{-1} \abs{t-s}^{1/2}.
\end{equation}

\begin{lemma}			
Let $\vec\Phi_\lambda(t,x)=(t, \vec \phi_\lambda(t,x))$ and $\vec\Psi_\lambda(t,x)=(t, \vec \psi_\lambda(t,x))$.
Let $I=(t_0-r^2,t_0)$ and $y_0=\vec \phi_\lambda(t_0, x_0)$.
For any $\alpha>0$, we have
\begin{align}			\label{eq1611mon}
\vec \Phi_\lambda(I \times B_{\alpha r}(x_0)) &\subset I \times B_{(\alpha(1+\varepsilon)+\varepsilon) r}(y_0),\\
					\label{eq1612mon}
\vec \Psi_\lambda(I \times B_{\alpha r}(y_0)) &\subset I \times B_{(1-\varepsilon)^{-1}(\alpha+\varepsilon) r}(x_0).
\end{align}
\end{lemma}
\begin{proof}
For any $(t,x) \in I \times B_{\alpha r}(x_0)$, it follows that
\begin{align*}
\abs{\vec \phi_\lambda(t,x)-y_0} &\le \abs{\vec \phi_\lambda(t,x)-\vec \phi_\lambda(t, x_0)}+\abs{\vec \phi_\lambda(t,x_0)-\vec \phi_\lambda(t_0, x_0)}\\
&\le (1+\varepsilon) \abs{x-x_0}+\varepsilon\abs{t-t_0}^{1/2} \le (1+\varepsilon)\alpha r+\varepsilon r,
\end{align*}
utilizing the inequalities \eqref{eq1359sat} and \eqref{eq1309sun}.
Consequently, this leads to
\[
\abs{\vec\Phi_\lambda(t,x)-(t,y_0)}_p < (1+\varepsilon)\alpha r+\varepsilon r,
\]
resulting in \eqref{eq1611mon}.
Similarly, taking into account that $\vec\psi(t_0, y_0)=x_0$, and utilizing the inequalities \eqref{eq1309sun} and \eqref{eq1448mon}, for any $(t,y) \in I \times B_{\alpha r}(y_0)$, we have
\begin{align*}
\abs{\vec \psi_\lambda(t,y)-x_0} &\le \abs{\vec \psi_\lambda(t,y)-\vec \psi_\lambda(t, y_0)}+\abs{\vec \psi_\lambda(t,y_0)-\vec \psi_\lambda(t_0, y_0)}\\
&\le (1-\varepsilon)^{-1} \abs{y-y_0}+\varepsilon(1-\varepsilon)^{-1}\abs{t-t_0}^{1/2} \\
&\le (1-\varepsilon)^{-1}\alpha r+\varepsilon(1-\varepsilon)^{-1}r.
\end{align*}
This implies \eqref{eq1612mon}.
\end{proof}

Observe that with $\vec \phi(t,x)=\vec \phi_\lambda(t,x)=x+\vec \xi_\lambda(t,x)$, we have
\begin{equation}			\label{eq2210mon}
\mathscr{L}\phi^k=\mathscr{L}(x_k+ \xi_\lambda^k)= b^k+ \mathscr{L} \xi_\lambda^k=b^k+ (\mathscr{L} -\lambda) \xi_\lambda^k + \lambda \xi_\lambda^k=\lambda \xi_\lambda^k.
\end{equation}
Therefore, by \eqref{eq2135mon} and \eqref{eq2210mon}, the identity \eqref{eq2207mon} becomes
\begin{multline}			\label{eq2224mon}
I=\iint_{\bR^{d+1}} u(t,x) \left\{\partial_t \eta(t,\vec\phi(t,x))+ \tilde a^{kl}(t,x) D_{kl} \eta(t,\vec\phi(t,x))+ \tilde b^k(t,x) D_k \eta(t,\vec\phi(t,x)) \right.\\
\left. -c(t,x)\eta(t,\vec\phi(t,x))  \right\}dxdt=0,
\end{multline}
where  $\tilde b^k=\lambda \xi^k_\lambda$.
Then, by the change of variables $x=\vec \psi(t,y)=\vec \psi_\lambda(t,y)$, defining
\begin{equation}			\label{eq2046sun}
\hat u(t,y):= u(t,\vec\psi(t,y)) \,\abs{\det D\vec\psi(t,y)}=u(t,x) \,\abs{\det D\vec\psi(t,\vec\phi(t,x))},
\end{equation}
and setting (recall definition \eqref{eq2137mon})
\begin{equation}			\label{eq1430sun}
\begin{aligned}
\hat a^{kl}(t,y) &:= \tilde a^{kl}(t,\vec \psi(t,y))=\tilde a^{kl}(t,x)=a^{ij}(t,x) D_i\phi^k(t,x) D_j \phi^l(t,x),\\
\hat b^{k}(t,y) &:= \tilde b^{k}(t,\vec \psi(t,y))=\tilde b^{k}(t,x)=\lambda \xi_\lambda^k(t,x),\\
\hat c(t,y) &:= c(t,\vec \psi(t,y))=c(t,x),
\end{aligned}
\end{equation}
we obtain
\[
0=\iint_{\bR^{d+1}} \hat u(t,y)\left\{\partial_t \eta(t,y)+ \hat a^{kl}(t,y) D_{kl} \eta(t,y)+ \hat b^k(t,y) D_k \eta(t,y)-\hat c(t,y)\eta(t,y)  \right\}dydt.
\]
Therefore, the identity \eqref{eq2224mon} becomes
\[
I=\iint_{\bR^{d+1}} \hat u (\partial_t +\hat a^{ijl}D_{ij} + \hat b^i D_i-\hat c) \eta =0,
\]
so that $\hat u$ is a solution to
\[
-\partial_t \hat u +D_{ij}(\hat a^{ij} \hat u) -D_i(\hat b^i \hat u)- \hat c \hat u=0\quad \text{in }Q':=\vec\Psi(Q).
\]

\begin{remark}			\label{rmk_b}
It is evident form \eqref{eq1430sun} that $\hat{\mathbf A}=(\hat a^{ij})$ satisfies Condition \ref{cond1}, as $\vec \phi(t,x)=x+\vec \xi_\lambda(t,x)$ and $\vec \xi_\lambda$ satisfies the estimate \eqref{eq1255mon}; see \cite[Lemma 3.4]{DEK21}.
We also note that $\hat{\vec b}=(\hat b^1,\ldots, \hat b^d)$ is bounded and it is Lipschitz in $x$.
Additionally, $\hat c$ is bounded and it is H\"older continuous in $x$ under the condition that $c$ possesses the same property.
This condition is certain satisfied when $c$ undergoes regularization as outlined in \ref{sec3.1}.
Invoking Remark~\ref{rmk_c}, we note that
\begin{gather*}
\sup_{\bR^{d+1}}\left(\abs{\hat{\vec b}}+ \abs{D \hat{\vec b}}\right) \le \lambda \sup_{\bR^{d+1}}\left(\abs{\vec \xi_\lambda}+ \abs{D \vec \xi_\lambda}\right) \le N,\\
\sup_{\bR^{d+1}}\,\abs{\hat c}\le \sup_{\bR^{d+1}}\,\abs{\bar c}\le N,\\
\sup_{t \in \bR} \, [\hat c(t,\cdot)]_{\mathscr{C}^{1-\beta}(\bR^d)} \le [\bar c]_{\mathscr{C}^{(1-\beta)/2,1-\beta}(\mathbb{R}^{d+1})}\, \norm{D \vec \psi}_{L_\infty(\bR^{d+1})}^{1-\beta}  \le N,
\end{gather*}
where $N$ depends only on $d$, $\delta$, $\omega_{\mathbf A}^{\mathsf x}$, $p$, $q$, $\beta$, $\mathfrak{b}$, $\mathfrak{c}$, and $\varepsilon \in (0,\frac12)$.
\end{remark}

\subsection{Absorbing lower order coefficients into the principal coefficients}	\label{sec3.3}
Let $u$ be a solution to
\[
\mathscr{L}^*u-cu=-\partial_t u+\sum_{i,j=1}^d D_{ij}(a^{ij}u)-\sum_{i=1}^d D_i(b^i u)-cu=0 \;\text{ in }\;Q\subset \bR^{d+1},
\]
where $\vec b=(b^1,\ldots, b^d)$ and $c$ are bounded functions that are H\"older continuous in $x$.
More precisely, we assume that there exist constants $\tilde{\mathfrak b}$ and $\tilde{\mathfrak c}$ such that
\begin{equation}			\label{eq2139sun}
\norm{\vec b}_{L_\infty(\bR^{d+1})} + \sup_{t\in \bR}\, [\vec b(t,\cdot)]_{\mathscr{C}^{1-\beta}(\bR^d)} \le \tilde{\mathfrak b}, \quad 
\norm{c}_{L_\infty(\bR^{d+1})} + \sup_{t\in \bR}\, [c(t,\cdot)]_{\mathscr{C}^{1-\beta}(\bR^d)} \le \tilde{\mathfrak c}.
\end{equation}

It is worth noting that by introducing a new variable $y\in \bR$, we can express
\[
-\sum_{i=1}^d D_i (b^i  u) - c u= -\frac12 \sum_{i=1}^d \partial_{y} \partial_{x_i} (y b^i  u)-\frac12\sum_{i=1}^d \partial_{x_i} \partial_y (y b^i  u) -\frac12 \partial_y \partial_y (y^2 cu).
\]
This leads us to define $\tilde{\mathbf A}=(\tilde a^{ij})$ as follows.
Let $\varrho$ be any positive number and let $\varphi:\bR \to \bR$ be a smooth function such that
\[
\varphi(y)=y\;\text{ for }\; \abs{y}\le \varrho, \quad \varphi(y)=\varrho+1 \;\text{ for }\; \abs{y} \ge \varrho+2,\quad \abs{\varphi'(y)} \le 1.
\]
For $(t,x,y)\in \bR\times \bR^d \times \bR=\bR^{d+2}$, we define
\begin{equation}				\label{eq1909tue}
\begin{aligned}
\tilde a^{ij}(t,x,y)&=a^{ij}(t,x) \qquad (i,j=1,\ldots d),\\
\tilde a^{i, d+1}(t,x,y)&=\tilde a^{d+1,i}(t,x,y)=- \tfrac12  \varphi(y) b^i(t,x) \qquad (i=1,\ldots, d),\\
\tilde a^{d+1,d+1}(t,x,y)&=-\tfrac12 \varphi(y)^2 c(t,x)+\mu,
\end{aligned}
\end{equation}
where $\mu$ is a constant to be fixed shortly.
Then, setting $x_{d+1}=y$, we deduce that $u(t,x,y)=u(t,x)$ satisfies
\[
-\partial_t  u + \sum_{i,j=1}^{d+1} D_{ij}(\tilde a^{ij} u)=0 \quad\text{in }\tilde Q:=Q \times (-\varrho,\varrho) \subset \bR^{d+2}.
\]

We will show that $\tilde{\mathbf A}$ satisfies Condition \ref{cond1}.
Utilizing \eqref{eq2139sun}, we obtain
\begin{align*}
\sum_{i,j=1}^{d+1} \tilde a^{ij} \xi_i \xi_j &= \sum_{i,j=1}^{d}  a^{ij} \xi_i \xi_j - \sum_{i=1}^d \varphi(y) b^i \xi_i \xi_{d+1} - \frac12 \varphi(y)^2 c \xi_{d+1}^2+\mu\xi_{d+1}^2 \\
&\ge \delta \sum_{i=1}^d \xi_i^2 - (\varrho+1) \tilde{\mathfrak b} \abs{\xi_{d+1}} \left(\sum_{i=1}^d \xi_i^2\right)^{1/2}-\frac12 (\varrho+1)^2 \tilde{\mathfrak c}  \xi_{d+1}^2 + \mu \xi_{d+1}^2\\
&\ge \delta \sum_{i=1}^{d+1} \xi_i^2 - \epsilon \sum_{i=1}^d \xi_i^2 - \frac{1}{4\epsilon} (\varrho+1)^2 \tilde{\mathfrak b}^2 \xi_{d+1}^2- \frac{1}{2} (\varrho+1)^2 \tilde{\mathfrak c}^2 \xi_{d+1}^2+\mu \xi_{d+1}^2,
\end{align*}
where we used the Cauchy's inequality.
Therefore, by taking $\epsilon=\frac12 \delta$, and then choosing
\[
\mu=\frac12 \delta+\frac{(\varrho+1)^2}{2\delta} \left(\tilde{\mathfrak b}^2+\delta \tilde{\mathfrak c}^2\right),
\]
we see that $\tilde{\mathbf A}=(\tilde a^{ij})$ satisfies
\[
\sum_{i,j=1}^{d+1} \tilde a^{ij}(t,x,y) \xi_i \xi_j \ge \frac{\delta}{2} \sum_{i=1}^{d+1}\xi_i^2,\quad \forall (t,x,y)\in \bR^{d+2},\;\;\forall \xi \in \bR^{d+1}.
\]
Moreover, it is clear that $\tilde a^{ij}$ are bounded on $\bR^{d+2}$.
Therefore, there exists $\tilde \delta \in (0,1]$ such that the eigenvalues of $\tilde{\mathbf A}$ lie in the interval $[\tilde \delta, 1/\tilde \delta]$, and $\tilde\delta$ is completely determined by $\delta$, $\tilde{\mathfrak b}$, $\tilde{\mathfrak c}$, and $\varrho$.
Also, it is evident from \eqref{eq1909tue} that $\tilde{\mathbf A}$ is of Dini mean oscillation in $(x,y)$, and $\omega_{\tilde{\mathbf A}}^{\mathsf x}$ is controlled by $\omega_{\mathbf A}^{\mathsf x}$, $\tilde{\mathfrak b}$, $\tilde{\mathfrak c}$, and $\varrho$.

\begin{conclusion}			\label{cond1prime}
We have verified that $\tilde{\mathbf A}$ satisfies Condition \ref{cond1} in $\bR^{d+2}$, with $\tilde \delta$ replacing $\delta$.
Moreover, we have observed that $\tilde\delta$ depends on $\delta$, $\tilde{\mathfrak b}$, $\tilde{\mathfrak c}$, and $\varrho$, and that $\omega_{\tilde{\mathbf A}}^{\mathsf x}$ is controlled by $\omega_{\mathbf A}^{\mathsf x}$, $\tilde{\mathfrak b}$, $\tilde{\mathfrak c}$, and $\varrho$.
\end{conclusion}

\begin{remark}			\label{rmk_a}
If $\vec b$ and $c$ are obtained via the regularization processes described in Section \ref{sec3.1} and \ref{sec3.2}, then $\tilde{\mathfrak b}$ and $\tilde{\mathfrak c}$ are both bounded by a constant $N$ depending by $d$, $\delta$, $\omega_{\mathbf A}^{\mathsf x}$, $p$, $q$, $\beta$, $\mathfrak{b}$, $\mathfrak{c}$, and $\varepsilon \in (0,\frac12)$. Refer to Remarks \ref{rmk_c} and \ref{rmk_b}.
\end{remark}

\subsection{Proof of Harnack inequality}
The subsequent proposition is an extension of Theorem~\ref{thm:harnack} to an operator with sufficiently regular lower order coefficients.
\begin{proposition}			\label{prop1}
Assume $\mathbf A$ satisfies Condition \ref{cond1}, and $\vec b$ and $c$ satisfy \eqref{eq2139sun}.
Let $R>0$ be a fixed number, $0<r<R/2$, and $(t_0, x_0) \in \bR^{d+1}$.
Suppose $u \in L_1(C_{2r})$ is a nonnegative weak solution of
\[
-\partial_tu +\dv^2(\mathbf{A} u)-\dv(\vec b u)-cu = 0 \;\text{ in }\;C_{2r}=C_{2r}(t_0, x_0).
\]
Then, we have
\[
\sup_{(t_0-3r^2, t_0-2r^2)\times B_{r}(x_0)} u \le N \inf_{(t_0-r^2, t_0)\times B_{r}(x_0)} u,
\]
where $N$ is a constant depending only on $d$, $\delta$, $\omega_{\mathbf A}^{\textsf x}$, $\tilde{\mathfrak b}$, $\tilde{\mathfrak c}$, and $R$.
\end{proposition}
\begin{proof}
We can assume that $(t_0,x_0)=(0,0)$.
Denote $\mathscr{C}_{r}=(-r^2, 0]\times \mathscr{B}_r$, where $\mathscr{B}_r$ is the $(d+1)$-dimensional Euclidean ball centered at the origin with radius $r$.
By the process described in Section~\ref{sec3.3}, with $\varrho=R$, we deduce that $u(t,x,y)=u(t,x)$ satisfies
 $-\partial_t u +\dv^2(\tilde{\mathbf A} u)=0$ in $\mathscr{C}_{2r}$.
Therefore, by Theorem~\ref{thm:harnack} and Conclusion~\ref{cond1prime}, we have
\[
\sup_{(-3r^2,-2r^2)\times \mathscr{B}_r} u \le N \inf_{(-r^2, 0)\times \mathscr{B}_r} u,
\]
where $N$ is a constant depending solely on $d$, $\delta$, $\omega_{\mathbf A}^{\textsf x}$, $\tilde{\mathfrak b}$, $\tilde{\mathfrak c}$, and $R$.
The proposition follows from the preceding inequality as $u$ is independent of $y$.
\end{proof}

We are now in the position to prove our main theorem, Theorem~\ref{thm_main}.
We can assume that $(t_0,x_0)=(0,0)$ and $R=1$.
By applying the regularization procedure in Section \ref{sec3.1}, we deduce that $\bar u =\zeta u \in L_1(C_{2r})$, and it is a weak solution of
\[
-\partial_t \bar u + \dv^2(\mathbf A \bar u)- \dv(\bar{\vec b} \bar u)-\bar c \bar u=0\;\text{ in }\;C_{2r},
\]
where $\bar{\vec b}$ and $\bar c$ satisfying the condition specified in Remark~\ref{rmk_c}.

As $1/2 \le \zeta \le 2$ by \eqref{eq1934sat}, we find that if $\bar u$ satisfies the Harnack inequality \eqref{harnack}, then $u$ satisfies the same inequality, with $4N$ replacing $N$.
Therefore, it is enough to consider the case when $c$ is bounded and Lipschitz continuous in $x$.

Next, recall the mappings $\vec\phi=\vec\phi_\lambda$ and $\vec\psi=\vec\psi_\lambda$ from Section \ref{sec3.2}.
We choose sufficiently large value for $\lambda$ to ensure the satisfaction of  \eqref{eq1309sun} and \eqref{eq1359sat} with $\varepsilon=1/3$, along with the fulfillment of \eqref{eq2118sat}.
Then,
\[
\hat u(t,y)= u(t,\vec\psi(t,y)) \,\abs{\det D\vec\psi(t,y)}=u(t,x) \,\abs{\det D\vec\psi(t,\vec\phi(t,x))},
\]
as defined in \eqref{eq2046sun}, is a weak solution of
\[
-\partial_t \hat u +D_{ij}(\hat a^{ij} \hat u) -D_i(\hat b^i \hat u)- \hat c \hat u=0\quad \text{in } \vec\Psi(C_{4r}),
\]
where $\hat a^{ij}$, $\hat b^i$, $\hat c$ are as defined in \eqref{eq1430sun}, and $\vec\Psi(t,y)=(t, \vec\psi(t,y))$.

By Remarks \ref{rmk_b} and \ref{rmk_a}, 
we deduce that the coefficients $\hat{\mathbf A}$, $\hat{\vec b}$, and $\hat c$ satisfy 
the hypothesis of Proposition \ref{prop1}.
It follows from \eqref{eq1611mon} that
\[
(-16r^2, 0] \times B_{2r}(y_0) \subset  \vec\Psi(C_{4r}),\quad y_0=\vec\phi_\lambda(0,0).
\]
In particular, we deduce that $C_{2r}(0,y_0) \subset \vec\Psi(C_{4r})$.
Therefore, by Proposition \ref{prop1}, we have
\[
\sup_{(-3r^2, -2r^2)\times B_{r}(y_0)} \hat{u} \le N \inf_{(-r^2, 0)\times B_{r}(y_0)} \hat{u}.
\]
From \eqref{eq2118sat} and \eqref{eq2046sun}, we obtain
\[
\sup_{\vec \Phi\left((-3r^2, -2r^2)\times B_{r}(y_0)\right)} u \le 4N \inf_{\vec \Phi\left((-r^2, 0)\times B_{r}(y_0)\right)} u.
\]
It is worth noting that from \eqref{eq1612mon}, we can infer the inclusion relationship:
\begin{align*}
(-3r^2, -2r^2) \times B_{r/3} &\subset \Phi\left((-3r^2, -2r^2)\times B_{r}(y_0)\right),\\
(-r^2, 0) \times B_{r/3} &\subset \Phi\left((-r^2, 0)\times B_{r}(y_0)\right).
\end{align*}
Certainly, the theorem is a straightforward consequence of the simple observation 
that the inclusion $A\subset B$ implies the inequalities
\[
\sup_{A} u \le \sup_{B}u\quad\text{and}\quad  \inf_{B} u \le \inf_{A}u.\eqno\qed
\]

\section{Harnack inequality for elliptic equations in double divergence form}		\label{sec4}
In this section, we discuss the Harnack inequality for elliptic equations in double divergence form.
Let $L$ be an elliptic operator in non-divergence form defined as follows:
\[
Lv:=a^{ij}D_{ij}v+b^i v.
\]
We are interested in the adjoint operator of $L$, denoted $L^*$, which is given by:
\[
L^*u=D_{ij}(a^{ij}u)-D_i(b^i u).
\]
Let $c$ be a real-valued function on $\mathbb{R}^d$.
Our objective is to present Harnack's inequality for nonnegative solutions $u$ to the elliptic equation in double divergence form:
\[
L^*u-cu=0.
\]

\begin{condition}			\label{cond1e}
The principal coefficient matrix $\mathbf A=(a^{ij})$ is symmetric and there is a constant $\delta \in (0,1]$ such that the eigenvalues of $\mathbf A(x)$ are in $[\delta, 1/\delta]$ uniformly in $x \in \mathbb{R}^{d}$.
Moreover, $\mathbf{A}$ is of Dini mean oscillation, that is,
\[
\int_0^{1} \frac{\omega_{\mathbf A}(t)}{t}\,dt <\infty,
\]
where the mean oscillation function $\omega_{\mathbf A}: \bR_+ \to \bR$ is defined by
\[
\omega_{\mathbf A}(r):=\sup_{x\in \mathcal{B}} \fint_{\mathcal{B} \cap B_r(x)} \,\abs{\mathbf A(y)-\bar {\mathbf A}_{x,r}}\,dy, \;\; \text{where }\bar{\mathbf A}_{x,r} :=\fint_{\mathcal{B} \cap B_r(x)} \mathbf A.
\]
\end{condition}

\begin{condition}			\label{cond2e}
The lower-order coefficients $\vec b=(b^1,\ldots, b^d)$ and $c$ 
are such that $\vec b \in L_{p_0, \rm loc}(\bR^d)$ 
and $c \in L_{p_0/2, \rm{loc}}(\bR^d)$ for some $p_0>d$.
\end{condition}

The following theorem serves as an elliptic counterpart 
to our main result, enhancing Theorem 3.5 in \cite{BRS23} by requiring 
only $c \in L_{p, \rm{loc}}$ for $p>d/2$.

\begin{theorem}		\label{thm2}
Assume Conditions \ref{cond1e} and \ref{cond2e} hold.
Let $R>0$ be a fixed number, $0<r<R/4$, and $x_0 \in \bR^{d}$.
Denote $B_{r}= B_{r}(x_0)$. 
Suppose $u \in L_1(B_{4r})$ is a nonnegative solution of
\[
L^* u - cu =0\;\;\text{ in }\;\; B_{4r}.
\]
Then, we have
\begin{equation}			\label{eq1424mon}
\sup_{B_{r/2}} u \le N \inf_{B_{r/2}} u,
\end{equation}
where $N$ is a constant depending only on $d$, $\delta$, $\omega_{\mathbf A}$, $p_0$, $R$, $\norm{\vec b}_{L_{p_0}(B_{R})}$, and $\norm{c}_{L_{p_0/2}(B_{R})}$.
\end{theorem}

\begin{proof}
We may assume that $p_0\le 2d$.
We consider $\zeta$ that is of the form $\zeta=1+\zeta_\lambda$, where $\zeta_\lambda$ is a solution to the problem
\begin{equation}			\label{eq1530tue}
(L-c-\lambda) \zeta_\lambda=c \;\text{ in }\; B_R,\quad \zeta_\lambda=0\;\text{ on }\;\partial B_R.
\end{equation}
Note that the assumptions for \cite[Theorem 2.5]{Krylov2023e} are satisfied with $\Omega = B_R$, $p=p_0/2$, $q_b=d$, and $q_c=p_0/2$.
Therefore, by \cite[Theorem 2.5]{Krylov2023e}, there exist constant $\lambda_0\ge 1$ and $N_0$, depending only on $d$, $\delta$, $\omega_{\mathbf A}$, $p_0$, $R$, $\norm{\vec b}_{L_{p_0}(B_R)}$, and $\norm{c}_{L_{p_0/2}(B_R)}$, such that for $\lambda \ge \lambda_0$, there exists a unique solution $\zeta_\lambda \in \mathring W^{2}_{p_0/2}(B_R)$ to the problem \eqref{eq1530tue},
satisfying the estimates
\[
\norm{D^2 \zeta_\lambda}_{L_{p_0/2}(B_R)}+\sqrt{\lambda}\norm{D\zeta_\lambda}_{L_{p_0/2}(B_R)}+\lambda\norm{\zeta_\lambda}_{L_{p_0/2}(B_R)} \le N_0 \norm{c}_{L_{p_0/2}(B_R)}.
\]
In particular, utiilizging $\lambda \ge \lambda_0 \ge 1$, we obtain
\begin{equation}				\label{eq1831tue}
\norm{\zeta_\lambda}_{W^2_{p_0/2}(B_R)} \le N_0 \norm{c}_{L_{p_0/2}(B_R)},\quad \norm{\zeta_\lambda}_{L_{p_0/2}(B_R)} \le N_0 \lambda^{-1} \norm{c}_{L_{p_0/2}(B_R)}. 
\end{equation}
Since $2p_0>d$, a well-known interpolation inequality yields (see \cite[Theorem 5.8]{Adams})
\[
\norm{\zeta_\lambda}_{L_\infty(B_R)} \le N_1 \norm{\zeta_\lambda}_{W^2_{p_0/2}(B_R)}^{d/2p_0}\norm{\zeta_\lambda}_{L_{p_0/2}(B_R)}^{1-d/2p_0},
\]
where $N_1$ depends only on $d$, $p_0$, and $R$.
Therefore, we deduce from \eqref{eq1831tue} that
\[
\norm{\zeta_\lambda}_{L_\infty(B_R)} \le N \lambda^{-1+d/2p_0} \norm{c}_{L_{p_0/2}(B_R)}.
\]
Taking $\lambda$ sufficiently large in the previous inequality, and utilizing the Sobolev embedding, we find that $\zeta=1+\zeta_\lambda$ satisfies
\begin{equation}				\label{eq1544tue}
1/2 \le \zeta \le 2\;\text{ in }\; B_R,\quad \norm{D \zeta}_{L_{p_0}(B_R)} \le N \norm{c}_{L_{p_0/2}(B_R)}.
\end{equation}
We note that $\lambda$ is determined by $d$, $\delta$, $\omega_{\mathbf A}$, $p_0$, $R$, $\norm{\vec b}_{L_{p_0}(B_R)}$, and $\norm{c}_{L_{p_0/2}(B_R)}$.

Then, similar to \eqref{eq2033mon}, the function $\tilde u:= \zeta u$ satisfies
\[
\int_{\bR^{d}} \tilde u \left(L + 2a^{ij}\frac{D_j \zeta}{\zeta} D_i + \lambda \frac{\zeta-1}{\zeta} \right) \eta =0,\quad \forall \eta \in C^\infty_c(B_{4r}).
\]
Therefore, if we define
\[
\tilde b^i:= b^i + 2a^{ij}\frac{D_j \zeta}{\zeta}\quad\text{and}\quad \tilde c:=-\lambda \frac{\zeta-1}{\zeta},
\]
then $\tilde u=\zeta u$ is a solution of
\[
D_{ij}(a^{ij} u)- D_i(\tilde b^i u)-cu=0 \text{ in }B_{4r}.
\]
Note that $\tilde c \in L_\infty(B_R)$ and $\tilde{\vec b} \in L_{p_0}(B_R)$.
Moreover, the following estimates hold:
\begin{equation}			\label{eq2125mon}
\norm{\tilde{\vec b}}_{L_{p_0}(B_R)} \le \norm{\vec b}_{L_{p_0}(B_R)} + 2\delta^{-1}\gamma^{-1}\norm{D \zeta}_{L_{p_0}(B_R)},\quad \norm{\tilde c}_{L_\infty(B_R)} \le \lambda.
\end{equation}
According to \cite[Theorem 3.5]{BRS23},  $\tilde u=\zeta u$ satisfies the Harnack inequality
\[
\sup_{B_{r/2}} \tilde u \le N \inf_{B_{r/2}} \tilde u,
\]
where $N$ depends only on $d$, $\delta$, $\omega_{\mathbf A}$, $p_0$, $R$, $\norm{\tilde{\vec b}}_{L_{p_0}(B_{R})}$, and $\norm{\tilde c}_{L_{p_0}(B_{R})}$.
Therefore, we derive \eqref{eq1424mon} from the previous inequality, along with \eqref{eq1544tue} and \eqref{eq2125mon}.
\end{proof}


\begin{thebibliography}{m}
\bibitem{Adams}
Adams, Robert A.; Fournier, John J. F.
\textit{Sobolev spaces}.
Pure Appl. Math. (Amst.), \textbf{140}.
Elsevier/Academic Press, Amsterdam, 2003.

\bibitem{Bauman84b}
Bauman, Patricia
\textit{Positive solutions of elliptic equations in nondivergence form and their adjoints}.
Ark. Mat. \textbf{22} (1984), no. 2, 153--173.

\bibitem{BKRS15}
Bogachev, Vladimir; Krylov, Nicolai.; R\"ockner, Michael; Shaposhnikov, Stanislav. \textit{Fokker-Planck-Kolmogorov equations}.
Mathematical Surveys and Monographs, \textbf{207}.
American Mathematical Society, Providence, RI, 2015.

\bibitem{BRS23}
Bogachev, Vladimir; R\"ockner, Michael; Shaposhnikov, Stanislav.
\textit{Zvonkin's transform and the regularity of solutions to double divergence form elliptic equations}.
Comm. Partial Differential Equations \textbf{48} (2023), no. 1, 119--149.

\bibitem{Brezis}
Brezis, Ha\"im.
\textit{On a conjecture of J. Serrin}. 
Atti Accad. Naz. Lincei Rend. Lincei Mat. Appl. \textbf{19} (2008), no.4, 335--338.

\bibitem{DK17}
Dong, Hongjie; Kim, Seick.
\textit{On $C^1$, $C^2$, and weak type-$(1,1)$ estimates for linear elliptic operators}.
Comm. Partial Differential Equations \textbf{42} (2017), no. 3, 417--435.

\bibitem{DEK18}
Dong, Hongjie; Escauriaza, Luis; Kim, Seick.
\textit{On $C^1$, $C^2$, and weak type-$(1,1)$ estimates for linear elliptic operators: part II}.
Math. Ann. \textbf{370} (2018), no. 1-2, 447--489.

\bibitem{DEK21}
Dong, Hongjie; Escauriaza, Luis; Kim, Seick.
\textit{On $C^{1/2,1}$, $C^{1,2}$, and $C^{0,0}$ estimates for linear parabolic operators}.
J. Evol. Equ. \textbf{21} (2021), no. 4, 4641--4702.

\bibitem{EM2017}
Escauriaza, Luis; Montaner, Santiago.
\textit{Some remarks on the $L^p$ regularity of second derivatives of solutions 
to non-divergence elliptic equations and the Dini condition}. Rend. Lincei Mat. Appl. \textbf{28} (2017), 49-63.

\bibitem{FGP2010}Flandoli, Franco; Gubinelli, Massimiliano; Priola, Enrico.
\textit{Well-posedness of the transport equation by stochastic perturbation}.
Invent. Math. \textbf{180} (2010), no. 1, 1--53.

\bibitem{Krylov1996}
Krylov, N. V.
\textit{Lectures on elliptic and parabolic equations in H\"older spaces}.
Graduate Studies in Mathematics, \textbf{12}. 
American Mathematical Society, Providence, RI, 1996.

\bibitem{Krylov2007}
Krylov, N. V.
\textit{Parabolic and elliptic equations with VMO coefficients}.
Comm. Partial Differential Equations \textbf{32} (2007), no. 1-3, 453--475.

\bibitem{Krylov2008}
Krylov, N. V.
\textit{Lectures on elliptic and parabolic equations in Sobolev spaces}.
Graduate Studies in Mathematics, \textbf{96}. 
American Mathematical Society, Providence, RI, 2008.

\bibitem{Krylov2023e}
Krylov, N. V.
\textit{Elliptic equations in Sobolev spaces with Morrey drift and the zeroth-order coefficients}.
Trans. Amer. Math. Soc. \textbf{376} (2023), no.10, 7329--7351.

\bibitem{Krylov2023}
Krylov, N. V.
\textit{On parabolic equations in Morrey spaces with VMO $a$ and Morrey  $b$, $c$}.
arXiv:2304.03736 [math.AP]

\bibitem{KS79}
Krylov, N. V.; Safonov, M. V.
\textit{An estimate for the probability of a diffusion process hitting a set of positive measure}. (Russian)
Dokl. Akad. Nauk SSSR \textbf{245} (1979), no. 1, 18--20.

\bibitem{KS80}
Krylov, N. V.; Safonov, M. V.
\textit{A property of the solutions of parabolic equations with measurable coefficients}. (Russian)
Izv. Akad. Nauk SSSR Ser. Mat. \textbf{44} (1980), no. 1, 161--175, 239.
 
\bibitem{LSU}
Lady\v{z}enskaja, O. A.; Solonnikov, V. A.; Ural'ceva, N. N.
\textit{Linear and quasilinear equations of parabolic type}.
American Mathematical Society: Providence, RI, 1967.

\bibitem{Moser61}
Moser, J\"urgen.
\textit{On Harnack's theorem for elliptic differential equations}.
Comm. Pure Appl. Math. \textbf{14} (1961), 577--591.

\bibitem{Moser64}
Moser, J\"urgen.
\textit{A Harnack inequality for parabolic differential equations}.
Comm. Pure Appl. Math. \textbf{17} (1964), 101--134.

\bibitem{Safonov80}
Safonov, M. V.
\textit{Harnack's inequality for elliptic equations and H\"older property of their solutions}. (Russian)
Zap. Nauchn. Sem. Leningrad. Otdel. Mat. Inst. Steklov. (LOMI) \textbf{96} (1980), 272--287, 312.

\bibitem{Sjogren73}
Sj\"ogren, Peter.
\textit{On the adjoint of an elliptic linear differential operator and its potential theory}.
Ark. Mat. \textbf{11} (1973), 153--165.

\bibitem{Sjogren75}
Sj\"ogren, Peter.
\textit{Harmonic spaces associated with adjoints of linear elliptic operators}.
Ann. Inst. Fourier (Grenoble) \textbf{25} (1975), no. 3-4, xviii, 509--518.

\bibitem{Zvonkin}
Zvonkin, A. K.
\textit{A transformation of the phase space of a diffusion process that removes the drift}. (Russian)
Mat. Sb. (N.S.) \textbf{93}(\textbf{135})(1974), 129--149, 152.
\end{thebibliography}
\end{document}